\documentclass[10pt]{amsart}
\usepackage{amssymb}
\usepackage{amsmath}
\usepackage{graphicx}
\usepackage{hyperref}
\hypersetup{
colorlinks=true,
linkcolor=blue,
filecolor=blue,
urlcolor=blue,
citecolor=blue,
pdfpagemode=Fullscreen
}
\usepackage{xcolor}
\usepackage{soul}

\newtheorem{thm}{Theorem}[section]
\newtheorem{lem}[thm]{Lemma}

\newtheorem{cor}[thm]{Corollary}
\newtheorem{Q}[thm]{Question}
\newtheorem{Def}[thm]{Definition}
\newtheorem{prop}[thm]{Proposition}
\newtheorem{rem}[thm]{Remark}
\newtheorem{ex}[thm]{Example}

\newcommand{\bdfn}{\begin{Def} \rm}
\newcommand{\edfn}{\end{Def}}

\newcommand{\beqa}{\begin{eqnarray*}}
\newcommand{\eeqa}{\end{eqnarray*}}

\newcounter{cnt1}
\newcounter{cnt2}
\newcounter{cnt3}
\newcounter{cnt4}
\newcommand{\blr}{\begin{list}{$($\roman{cnt1}$)$} {\usecounter{cnt1}
 \setlength{\topsep}{0pt} \setlength{\itemsep}{0pt}}}
\newcommand{\blR}{\begin{list}{\Roman{cnt4}.\ } {\usecounter{cnt4}
 \setlength{\topsep}{0pt} \setlength{\itemsep}{0pt}}}
\newcommand{\bla}{\begin{list}{$(\alph{cnt2})$} {\usecounter{cnt2}
 \setlength{\topsep}{0pt} \setlength{\itemsep}{0pt}}}
\newcommand{\bln}{\begin{list}{$($\arabic{cnt3}$)$} {\usecounter{cnt3}
 \setlength{\topsep}{0pt} \setlength{\itemsep}{0pt}}}
\newcommand{\el}{\end{list}}

\begin{document}

\title[\tiny{Geometry  of ${\mathcal L}(X,Y^*)$}]{Geometry of the unit ball of ${\mathcal L}(X,Y^*)$}

\author[Rao, Seal]{T. S. S. R. K. Rao$^{1}$,\ Susmita Seal$^{2}$}

\address{{$^{1}$} Department of Mathematics, Shiv Nadar Institution of Eminence. Gautam Buddha Nagar-201314, India}
\email{srin@fulbrightmail.org}

\address {{$^{2}$} School of Mathematical Sciences, National Institute of Science Educational and Research Bhubaneswar, An OCC of Homi Bhabha National Institute, P.O. - Jatni, District - Khurda, Odisha - 752050, India}
		\email{susmitaseal1996@gmail.com}

\subjclass[2020]{Primary 46A22, 46B10, 46B25; Secondary 46B22, 47L05.
\hfill \textbf{\today}
}

\keywords{Projective tensor product, spaces of operators, geometry of Banach spaces, Namioka points }

\maketitle
\begin{abstract}
In this work we study the geometry of the unit ball of the space of operators ${\mathcal L}(X,Y^*)$, by considering the projective tensor product $X\hat{\otimes}_{\pi} Y$ as a predual. We prove that if an elementary tensor (rank one operator) of the form $x_0^*\otimes y_0^* $ in the unit sphere $ S_{{\mathcal L}(X,Y^*)}$ is a weak$^*$-strongly extreme point of the unit ball, then $x_0^*$ is weak$^*$-strongly extreme point  of unit ball of $X^*$ and $y_0^*$ is weak$^*$-strongly extreme point of the unit ball of  $Y^*$. We show that a similar conclusion holds if the rank one operator is a Namioka point (equivalently, point of weak$^*$-weak continuity for the identity mapping) on the unit sphere of ${\mathcal L}(X,Y^*)$. We also study extremal phenomenon in the unit ball of ${\mathcal L}(X,Y^*)^*$. We partly solve the open problem, when does an elementary tensor, whose components are Namioka points is again a Namioka point? We show that if a point $z\in S_{{\mathcal L}(X,Y^*)^*}$ is a weak$^*$-strongly extreme point of the unit ball, then $z=x\otimes y$ for some weak$^*$-strongly extreme points $x\in S_X$ and $y\in S_Y$, provided the space of compact operators, $\mathcal{K}(X,Y^*)$ is separating for $X\hat{\otimes}_{\pi} Y$.
\end{abstract}
\maketitle

\section{Introduction}
Let $X$ be a real Banach space. We  denote the closed unit ball of a Banach space $X$ by $B_X$, the unit sphere by $S_X.$ We first recall some geometric notations that we will be used in this study. 
Let $C$ be a bounded subset of $X,$ $x^* \in X^*$ and $\alpha > 0.$
Then the set
$$S(C, x^*, \alpha) = \{x \in C : x^*(x) > \sup x^*(C) - \alpha \}$$
 is called the
slice determined by $x^*$ and $\alpha.$
 Without loss of generality, for any slice $S(C,x^*,\alpha)$, we can consider the determining functional $x^* \in S_{X^*}$.
 Analogously one can define weak$^*$-slices in $X^*$ by considering the determining functional from the predual space (instead of dual space). Geometric notions of denting points and weak$^*$-denting points (see \cite{DU}) can be described in terms of slices of the unit ball. We say that $x^* \in S_{X^\ast}$ is a weak$^*$-denting point if and only if  weak$^*$-slices containing $x^*$ form a base for the norm neighbourhood system of $x^\ast$. We refer the monographs of D. Bourgin \cite{B1} and J. Diestel and J. J. Uhl \cite{DU} for basic definitions from the geometry of the unit ball of a Banach space, for vector-valued integration theory and tensor product spaces. The monograph \cite{L} by H. E. Lacey is a standard reference for isometric theory of Banach spaces and \cite{H}, \cite{D}  for general theory of Banach spaces.
 \vskip 1em
 Consider the identity mapping, $i: (S_{X^*},weak^*) \rightarrow (S_{X^*},weak)$. A point of continuity (if it exists) is called a Namioka point, see \cite{G} and the recent paper \cite{DMR}. In the case the range space is equipped with the norm topology, we call the point of continuity as weak$^*$-norm PC. The notion of weak-norm PC in $S_X$ is similarly defined.  It is known that a point in the dual unit ball is weak$^*$-denting if and only if it is a weak$^*$-norm PC and an extreme point, see \cite{LLT}.
While the notions of points of  continuity is classical, here we study its stability
in projective tensor product. In particular, we  study the stability of the continuity
of a elementary tensor in the dual of a projective tensor product space and compare its relation to the continuity of the component vectors in the respective unit spheres. See the recent paper \cite{P} where the authors have shown that under the additional assumption that the component spaces have the approximation property, the set of elementary tensors is a weakly closed set.
\vskip 1em

 Let $B(X \times Y )$ be the space of bilinear forms on $X \times Y$. We recall from \cite{R2} that the tensor product of
$X$ and $Y$ is denoted by $X\otimes Y$ and it is the subspace of the algebraic dual of $B(X \times Y )$ spanned by all elements $x\otimes y$ where $x\in X$ and $y\in Y.$  Given the tensor product $X\otimes Y$, the projective norm $\|.\|_{\pi}$ on $X\otimes Y$ is determined by
 $$\|u\|_{\pi}= \inf \{ \sum_{i=1}^{n} \|x_i\|  \|y_i\| : u=\sum_{i=1}^{n} x_i \otimes y_i\}  $$
 The pair $(X\otimes Y, \|.\|_{\pi})$ is sometimes denoted by $X \otimes_{\pi} Y$ and its completion is called projective tensor product, denoted by $X \hat{\otimes}_{\pi} Y.$ We have from \cite[Chapter VIII, Corollary 2]{DU}
  that the space $ {\mathcal L}(X, Y^*)$ is linearly isometric to the dual of $X\hat{\otimes}_\pi Y$. Here we always consider the weak$^*$-topology on ${\mathcal L}(X,Y^*)$ coming from the projective tensor product of $X,Y$.
 \vskip 1em
We start by recalling a classical theorem where ${\mathcal L}(X,Y^\ast)$ has a unique predual and the weak$^*$-topology is uniquely determined. If $Y$ is such that $Y^\ast=C(\Omega)$ for a compact set $\Omega$, then it follows that $\Omega$ is a hyperstonean space \cite[Theorem 11, page 96]{L}
 and for a positive measure $\mu$, $Y = L^1(\mu)$. Now by
 \cite[Chapter VIII, Example 10]{DU} we see that the space of Bochner integrable functions, $L^1(\mu,X)$ is a predual of ${\mathcal L}(X,Y^*)$. When $X$ is reflexive, that
 $L^1(\mu,X)$ is the unique predual of ${\mathcal L}(X, C(\Omega))$ follows from the results in \cite{C} and \cite{R}.
 Hence in this case the weak$^\ast$-topology on ${\mathcal L}(X, C(\Omega))$ is uniquely determined.
 \vskip 1em
 We now briefly recall some literature on extremal structures preserved by projective tensor products.
   It is easy to see that $B_{X \hat{\otimes}_{\pi} Y}= \overline{\mathrm{co}}(B_X \otimes B_Y)$ ($\mathrm{co}$ denotes the convex hull and the closure is taken in the norm topology). However, it is still an open problem, whether every extreme point of $B_{X \hat{\otimes}_{\pi} Y}$ must be of the form $x\otimes y$, for $x\in B_X$ and $y\in B_Y$.
Let  $\mathcal{K}_{w^*}(X^*,Y)^*$ denotes the space of all weak$^*$-weak continuous compact operators. It is known from \cite{CR} and \cite{RS1} that, the extreme points of the unit ball, $\mathrm{ext} (B_{\mathcal{K}_{w^*}(X^*,Y)^*})= \mathrm{ext} (B_{X^*}) \otimes \mathrm{ext} (B_{Y^*}),$ are preserved under this operation. As a consequence, if $X^*$ or $Y^*$ has RNP and $X^*$ or $Y^*$ has approximation property (see \cite[Chapter VIII]{DU}) then $\mathrm{ext} (B_{X^* \hat{\otimes}_{\pi} Y^*})= \mathrm{ext} (B_{X^*}) \otimes \mathrm{ext} (B_{Y^*}).$
    The situation is better  for the stronger notions of denting points and strongly exposed points (see \cite{DU} for the definitions). Very recently, in \cite{LGCZ}, the notions of preserved extreme point and weak-strongly exposed point in projective tensor products were studied and under some additional assumptions such extreme points turn out to be of the type $x \otimes y$, where the component vectors have the same extremal property. This is also partly the theme of this paper.
   \vskip 1em

   Some of the classical results include the work of Ruess and Stegall \cite{RS}, who proved that $\mathrm{strexp} (B_{X \hat{\otimes}_{\pi} Y})= \mathrm{strexp} (B_X) \otimes \mathrm{strexp} (B_Y),$ where $\mathrm{strexp}(A)$ denotes the set of all strongly exposed points of $A$.  Also for denting points, D. Werner \cite{W} proved an analogous result in a more general context. He showed
   $\mathrm{dent} (\overline{\mathrm{co}}(C\otimes D))= \mathrm{dent} (C) \otimes \mathrm{dent} (D),$ where $C\subset X$ and $D\subset Y$ are absolutely convex, bounded and closed subsets and $\mathrm{dent}(A)$ denotes the set of all denting points of $A$. As a consequence of Werner's result, it is true that if $x_0\in S_X$ and $y_0\in S_Y$, then $x_0\otimes y_0$ is weak$^*$-denting point of $B_{(X\hat{\otimes}_{\pi} Y)^{**}}$ if and only if $x_0$ is weak$^*$-denting point of $B_{X^{**}}$ and $y_0$ is weak$^*$-denting point of $B_{Y^{**}}$.
\vskip 1em

 Motivated by the above results in the case of projective tensors, in this note we also study the notion of weak$^*$-strongly extreme point, which can be seen as an amalgam of two notions, namely, Namioka point  and extreme. In particular this study is a comparison of extremal structures with respect to two locally convex topologies. Abstract studies of this nature have been also carried out in \cite{F} and \cite{M}.

 \begin{Def}
 Let $X$ be a Banach space. We recall from \cite{BB} and  \cite{G}:
a point $x_0^*\in S_{X^*}$ is called
\begin{enumerate}
\item  weak$^*$-strongly extreme point of $B_{X^*}$ if the family of weak$^*$-slices containing $x_0^*$ forms a local base for the weak topology of $X^*$ at $x_0^*$.
\item  $ x_0^\ast \in B_{X^*}$ is a Namioka point, if the identity map from $(B_{X^{*}}$, weak$^*)$ to $(B_{X^{*}}$, weak$)$ is continuous at $x_0^*$.
\end{enumerate}
\end{Def}

\begin{thm}\label{bb}
\cite{BB}
A point $x_0^*\in S_{X^*}$ is weak$^*$-weak PC and extreme point of $B_{X^*}$ if and only if it is weak$^*$-strongly extreme point of $B_{X^*}$.
\end{thm}

In Section 2,
we show that if
  $x^\ast \otimes y^\ast$ is a Namioka point  of $B_{{\mathcal L}(X,Y^*)}$ for some $x^*\in S_{X^*}$ and $y^*\in S_{Y^*}$, then the component functionals have the same property. Theorem 2.4 demonstrates the converse implication under some additional hypothesis on the Banach spaces and component functionals.
We always consider a Banach space $X$ as canonically embedded in its bidual $X^{\ast\ast}$. Since $B_{X}$ is weak$^\ast$-dense in $B_{X^{**}}$, we see that a Namioka point $\tau$ of $B_{X^{**}}$ is the canonical image of $ x \in S_X$. In this case we say that $x$ is a Namioka point of $S_X$. Let $x$, $y$ be unit vectors and consider $x \otimes y$ as an element of ${\mathcal L}(X,Y^*)^*$ under the canonical embedding. We show in Section 3, that if $x \otimes y$ is a Namioka point  of the identity mapping on the unit sphere, then the same is true of the components $x$, $y$ when they are considered as vectors in the corresponding biduals. As an application, we show that if a point $z\in S_{{\mathcal L}(X,Y^*)^*}$ is a weak$^*$-strongly extreme point of the unit ball, then $z=x\otimes y$ for some weak$^*$-strongly extreme points $x\in S_X$ and $y\in S_Y$, provided the space of compact operators, $\mathcal{K}(X,Y^*)$, is separating for $X\hat{\otimes}_{\pi} Y$.
 In Section 4, a similar conclusion is obtained for points of weak-norm continuity, improving on the result of Werner \cite{W}. We also give examples of spaces where no compact operator is a point of weak$^*$-norm PC in $S_{{\mathcal L}(X,Y^\ast)}$.
 In Section 5, we show that the weak density of the set of Namika points  of $B_{X^*}$ in $S_{X^*}$ does not imply that this is the case for any closed subspace $Y \subset X$ , answering a question raised in \cite{Ra3}.  \vskip 1em

However, a main question in this domain remains open.
\begin{Q}
If $S_{{\mathcal L}(X,Y^\ast)}$ has a Namioka point, then does the ball always have a rank one operator as a Namioka point ?
\end{Q}

\section{Stability of Namioka points in $B_{{\mathcal L}(X,Y^*)}$}

Our first result shows that if an operator of rank one is a Namioka point  of $B_{{\mathcal L}(X,Y^*)}$, then the same holds for component vectors.

\begin{thm}\label{ma1}
Let $x_0^*\in S_{X^*}$ and $y_0^*\in S_{Y^*}$. Suppose $x_0^*\otimes y_0^* \in S_{{\mathcal L}(X,Y^*)}$ is a Namioka point  of $B_{{\mathcal L}(X,Y^*)}$. Then $x_0^*$ is a Namioka point of $B_{X^*}$ and $y_0^*$ is a Namioka point of $B_{Y^*}$.
\end{thm}

\begin{proof}
Suppose $x^\ast \otimes y^\ast$ is a Namioka point. Consider two weakly open sets
$$\bigcap\limits_{i=1}^{n} S(B_{X^*}, x_i^{**}, \alpha_i)\  \mathrm{and} \ \bigcap\limits_{j=1}^{m} S(B_{Y^*}, y_j^{**}, \beta_j)$$ containing $x_0^*$ and $y_0^*$ respectively. Choose $0<\varepsilon <\min\{x_i^{**}(x_0^*)y_j^{**}(y_0^*): 1\leqslant i \leqslant n, 1\leqslant j \leqslant m \}$ such that for each $1\leqslant i \leqslant n$ and $1\leqslant j \leqslant m$  $$x_i^{**}(x_0^*) - \frac{\varepsilon}{y_j^{**}(y_0^*)}> 1-\alpha_i  \quad \mathrm{and} \quad y_j^{**}(y_0^*) - \frac{\varepsilon}{x_i^{**}(x_0^*)}> 1-\beta_j.$$ Also let $\gamma_{ij} := 1-x_i^{**}(x_0^*)y_j^{**}(y_0^*) + \varepsilon$. Consider weakly open set $$W:=\bigcap\limits_{i=1}^{n}\bigcap\limits_{j=1}^{m}S(B_{{\mathcal L}(X,Y^*)}, x_i^{**}\otimes y_j^{**}, \gamma_{ij}).$$ We observe that by hypothesis, $x_0^*\otimes y_0^* \in W.$ Thus there exists a weak$^*$-open set $\bigcap\limits_{k=1}^{p}S(B_{{\mathcal L}(X,Y^*)}, \tau_k, \delta_k)$ such that $x_0^*\otimes y_0^* \in \bigcap\limits_{k=1}^{p}S(B_{{\mathcal L}(X,Y^*)}, \tau_k, \delta_k) \subset W,$ where  $\tau_k\in S_{{\mathcal L}(X,Y^*)_*}=S_{(X\hat{\otimes}_{\pi} Y)}$ for all $k=1,\ldots,p$.
Without loss of generality, let $\tau_k= \sum\limits_{l=1}^{r(k)}\lambda_{k,l} (x_{k,l}\otimes y_{k,l})$ with $\lambda_{k,l}>0$ and  $\sum\limits_{l=1}^{r(k)}\lambda_{k,l}=1$. Then
\begin{equation}\notag
\begin{split}
(x_0^*\otimes y_0^*)(\sum\limits_{l=1}^{r(k)}\lambda_{k,l} (x_{k,l}\otimes y_{k,l}))> 1-\delta_k \ \forall 1\leqslant k\leqslant p\\
\Rightarrow \sum\limits_{l=1}^{r(k)}(\lambda_{k,l} y_0^*(y_{k,l})) x_0^*(x_{k,l})> 1-\delta_k \ \forall 1\leqslant k\leqslant p.
\end{split}
\end{equation}
Hence $\sum\limits_{l=1}^{r(k)} \eta_{k,l} x_0^*(x_{k,l})> 1-\delta_k$, where $\eta_{k,l} := \lambda_{k,l} y_0^*(y_{k,l})$ for all $k=1,\ldots, p$. Thus $$x_0^*\in \bigcap\limits_{k=1}^{p}S(B_{X^*}, \sum\limits_{l=1}^{r(k)} \eta_{k,l} x_{k,l}, \|\sum\limits_{l=1}^{r(k)} \eta_{k,l} x_{k,l}\|-1+\delta_k).$$
Let $x^*\in \bigcap\limits_{k=1}^{p}S(B_{X^*}, \sum\limits_{l=1}^{r(k)} \eta_{k,l} x_{k,l}, \|\sum\limits_{l=1}^{r(k)} \eta_{k,l} x_{k,l}\|-1+\delta_k)
.$ Then $$x^*\otimes y_0^* \in  \bigcap\limits_{k=1}^{p}S(B_{{\mathcal L}(X,Y^*)}, \tau_k, \delta_k) \subset W.$$
Thus $x_i^{**}(x^*)> \frac{1-\gamma_{i,j}}{y_j^{**}(y_0^*)}> 1-\alpha_i$ for all $i=1,\ldots,n$.
Hence $$x_0^*\in \bigcap\limits_{k=1}^{p}S(B_{X^*}, \sum\limits_{l=1}^{r(k)} \eta_{k,l} x_{k,l}, \|\sum\limits_{l=1}^{r(k)} \eta_{k,l} x_{k,l}\|-1+\delta_k) \subset \bigcap\limits_{i=1}^{n} S(B_{X^*}, x_i^{**}, \alpha_i).$$
Similarly, $$y_0^*\in \bigcap\limits_{k=1}^{p}S(B_{Y^*}, \sum\limits_{l=1}^{r(k)} \zeta_{k,l} x_{k,l}, \|\sum\limits_{l=1}^{r(k)} \zeta_{k,l} x_{k,l}\|-1+\delta_k) \subset \bigcap\limits_{j=1}^{m} S(B_{Y^*}, y_j^{**}, \beta_j),$$ where $\zeta_{k,l} := \lambda_{k,l} x_0^*(x_{k,l})$ for all $k=1,\ldots, p$. Consequently, $x_0^*$ and $y_0^*$ are Namioka points of $B_{X^*}$ and $B_{Y^*}$ respectively.
\end{proof}

\begin{cor}
Let $x_0^*\in S_{X^*}$ and $y_0^*\in S_{Y^*}$. Also let $x_0^*\otimes y_0^* \in S_{{\mathcal L}(X,Y^*)}$ be weak$^*$-strongly extreme point of $B_{{\mathcal L}(X,Y^*)}$. Then $x_0^*$ is weak$^*$-strongly extreme point of $B_{X^*}$ and $y_0^*$ is weak$^*$-strongly extreme point of $B_{Y^*}$.
\end{cor}

\begin{proof}
Observe that $x_0^*\otimes y_0^*$ is extreme point of $B_{{\mathcal L}(X,Y^*)}$ gives $x_0^*$ and $y_0^*$ are extreme points of $B_{X^*}$ and $B_{Y^*}$ respectively. Rest follows from Theorems $\ref{bb}$ and $\ref{ma1}$.
\end{proof}

We recall that investigation of Namioka points  is well connected with the uniqueness of the Hahn-Banach extensions, see \cite{G}. 

\begin{lem}\label{l1}
 \cite{G}
    Let $X$ be a Banach space. Then, $x^*\in S_{X^*}$ is a Namioka point if and only if $x^*$ has a unique norm preserving extension to $X^{**}$.
 \end{lem}
In the following theorem in order to show that an elementary tensor is a Namioka point, we consider  different additional conditions on the second component. We recall the abstract notion of a unitary from \cite{BJR}, $y_0 \in S_{Y}$ is said to be a unitary if $Y^\ast = \mathrm{span}\{y^\ast \in B_{Y^*}: y^*(y_0)=1\}$. It follows from the results in \cite{BJR} that any unitary is a strongly extreme point and  remain unitary in the bidual. As noted in Remark 3.10 in \cite{BJR}, any Banach space can be equivalently renormed  to have a unitary.
In the finite-dimensional setting, a complete classification is given for when an elementary tensor is a unitary (or equivalently, a vertex) \cite[Theorem 2.2]{GW}.
We will also use class of Banach spaces $Y$ for which ${\mathcal K}(X,Y)= {\mathcal L}(X,Y)$. See Theorem VIII.4 in \cite{DU} and its proof. We note that in the following theorem, the assumptions on the space $Y$ are invariant under equivalent renorming.
\begin{thm}\label{c1}
Let $Y$ be a reflexive Banach space with the approximation property (AP). Suppose $y_0^*\in S_{Y^*}$ is a unitary. Then for any Banach space $X$ such that ${\mathcal L}(X,Y^*)={\mathcal K}(X,Y^*)$ and for a Namioka point, $x_0^* \in S_{X^*}$, we have $x_0^*\otimes y_0^* \in S_{{\mathcal L}(X,Y^*)}$ is a Namioka point.

\end{thm}

\begin{proof}
Let  $S_{y_0^*}=\{y\in S_{Y} : y_0^*(y)=1\}$. Since $y_0^*$ is unitary, $\mathrm{span}{(S_{y^*_0})}=Y$ .
 Since $Y$ is reflexive and has the AP, by our assumptions on $X$,
$ \mathcal{L}(X,Y^*)^{*}= X^{**}\hat{\otimes}_{\pi} Y$.
We will use the equivalent formulation of a Namioka point in terms of unique extensions.  Let $T_1$ and $T_2$ be two  Hahn-Banach extensions of $x_0^*\otimes y_0^*$, i.e., for  $i=1,2$, we have
\begin{center}
$T_i: {\mathcal L}(X,Y^*)^{*} \rightarrow \mathbb{R}$ such that $T_i\Big|_{X\hat{\otimes}_{\pi} Y}= x_0^*\otimes y_0^*$ and $\|T_i\|=1$
\end{center}
For each $i=1,2$ and $y \in S_{y^*_0}$, we define
\begin{center}
$\eta_{i,y}: X^{**}\rightarrow \mathbb{R}$ \ by \ $\eta_{i,}(x^{**})=T_i(x^{**}\otimes y)$
\end{center}
Observe that
\begin{equation}\notag
\eta_{i,y}(x)=T_i(x\otimes y)= (x_0^*\otimes y_0^*)(x\otimes y)=x_0^*(x) y_0^*(y)= x_0^*(x) \quad \forall x\in X
\end{equation}
Thus $\eta_{i,y}\Big|_{X}=x_0^*$. Also $\|\eta_{i,y}\|=1$. Indeed, $\|\eta_{i,y}\| \leqslant \|T_i\|=1$. Choose $x_n \in S_{X}$ with $\lim\limits_{n\rightarrow \infty} x_0^*(x_n)=1$. Then $\|\eta_{i,y}\| \geqslant \lim\limits_{n\rightarrow \infty} \eta_{i,y}(x_n)=\lim\limits_{n\rightarrow \infty} x_0^*(x_n)=1$. Hence $\eta_{i,y}$ ($i=1,2$ and $k=1,\ldots,n$) are Hahn-Banach extensions of $x_0^*$. Since $x_0^*$ is a Namioka point,
\begin{equation}\notag
\begin{split}
\eta_{1,y}=\eta_{2,y}  \hspace{5 cm}\\
\Rightarrow \eta_{1,y} (x^{**})=\eta_{2,y}(x^{**})  \ \forall x^{**}\in X^{**} \hspace{2.7 cm}\\
\Rightarrow T_1(x^{**} \otimes y)= T_2(x^{**} \otimes y) \quad \forall y \in S_{y^*_0},\ \forall x^{**}\in X^{**}\\
\Rightarrow T_1=T_2 \hspace{6.5 cm}
\end{split}
\end{equation}
 Hence, $x_0^*\otimes y_0^*$ has unique Hahn-Banach extension to ${\mathcal L}(X,Y^*)^{*}$. Thus, $x_0^* \otimes y_0^*$ is a Namioka point of $B_{{\mathcal L}(X,Y^*)}$.
\end{proof}

Taking into account the fact that $(X/Y)^*=Y^{\perp}$, following proposition ensures the existence of unitaries in finite dimensions.
\begin{prop}
Let $X$ be a finite dimensional Banach space. Then for any $x\in S_X$ either $x$ is unitary in $X$ or $\pi(x)$ is a unitary in $X/Y$, where $\pi :X \rightarrow X/Y$ is the quotient map and $Y=\{x^*\in B_{X^*} : x^*(x)=1\}^{\perp}$.
\end{prop}


\section{Stability in higher duals of projective tensor products}

The following proof uses techniques similar  to the ones used during the proof of Theorem $\ref{ma1}$. For the sake of completeness we give the details of the proof here.

\begin{thm} \label{stab w*-w pc}
Let $x_0\in S_X$, $y_0\in S_Y$ and $x_0 \otimes y_0$ be a Namioka point of $B_{{\mathcal L}(X,Y^*)^{*}}$. Then $x_0$ is a Namioka point of $B_{X^{**}}$ and $y_0$ is a Namioka point of $B_{Y^{**}}$.
\end{thm}

\begin{proof}

Suppose $x_0 \otimes y_0$ is a Namioka point. Consider two weakly open sets
$$\bigcap\limits_{i=1}^{n} S(B_{X^{**}}, x_i^{***}, \alpha_i)\  \mathrm{and} \ \bigcap\limits_{j=1}^{m} S(B_{Y^{**}}, y_j^{***}, \beta_j)$$ containing $x_0$ and $y_0$ respectively. Choose $0<\varepsilon <\min\{x_i^{***}(x_0) y_j^{***}(y_0): 1\leqslant i \leqslant n, 1\leqslant j \leqslant m \}$ such that for each $1\leqslant i \leqslant n$ and $1\leqslant j \leqslant m$  $$x_i^{***}(x_0) - \frac{\varepsilon}{y_j^{***}(y_0)}> 1-\alpha_i  \quad \mathrm{and} \quad y_j^{***}(y_0) - \frac{\varepsilon}{x_i^{***}(x_0)}> 1-\beta_j.$$ Also let $\gamma_{ij} := 1-x_i^{***}(x_0)y_j^{***}(y_0) + \varepsilon$. Consider the weakly open set $$W:=\bigcap\limits_{i=1}^{n}\bigcap\limits_{j=1}^{m}S(B_{{\mathcal L}(X,Y^*)^{*}}, \psi_{i,j}, \gamma_{ij}),$$
where $\psi_{i,j}$ is a norm preserving extension of $x_i^{***}\otimes y_j^{***}$ from $X^{**}\otimes Y^{**}$ to ${\mathcal L}(X,Y^*)^{*}$. Observe that $x_0\otimes y_0 \in W.$ Then there exists a weak$^*$-open set $\bigcap\limits_{k=1}^{p}S(B_{{\mathcal L}(X,Y^*)^{*}}, T_k, \delta_k)$ such that $x_0\otimes y_0 \in \bigcap\limits_{k=1}^{p}S(B_{{\mathcal L}(X,Y^*)^{*}}, T_k, \delta_k) \subset W,$ where $T_k\in S_{{\mathcal L}(X,Y^*)}$ for all $k=1,\ldots,p$.

Then
\begin{equation}\notag
\begin{split}
(x_0\otimes y_0)(T_k)> 1-\delta_k \ \forall 1\leqslant k\leqslant p\\
\Rightarrow (y_0 \circ T_k)(x_0)> 1-\delta_k \ \forall 1\leqslant k\leqslant p.
\end{split}
\end{equation}
 Thus $$x_0\in \bigcap\limits_{k=1}^{p}S(B_{X^{**}}, y_0 \circ T_k, \|y_0 \circ T_k\|-1+\delta_k).$$
Let $x^{**}\in \bigcap\limits_{k=1}^{p}S(B_{X^{**}}, y_0 \circ T_k, \|y_0 \circ T_k\|-1+\delta_k).$
Then $$x^{**}\otimes y_0 \in  \bigcap\limits_{k=1}^{p}S(B_{{\mathcal L}(X,Y^*)^{*}}, T_k, \delta_k) \subset W.$$
Thus $x_i^{***}(x^{**})> \frac{1-\gamma_{i,j}}{y_j^{***}(y_0)}> 1-\alpha_i$ for all $i=1,\ldots,n$.
Hence $$x_0\in \bigcap\limits_{k=1}^{p}S(B_{X^{**}}, y_0 \circ T_k, \|y_0 \circ T_k\|-1+\delta_k) \subset \bigcap\limits_{i=1}^{n} S(B_{X^{**}}, x_i^{***}, \alpha_i).$$ Similarly, $$y_0\in \bigcap\limits_{k=1}^{p}S(B_{Y^{**}}, T_k(x_0), \|T_k(x_0)\|-1+\delta_k) \subset \bigcap\limits_{j=1}^{m} S(B_{Y^{**}}, y_j^{***}, \beta_j).$$
 Consequently, $x_0$ and $y_0$ are weak$^*$-weak PC of $B_{X^{**}}$ and $B_{Y^{**}}$ respectively.
\end{proof}

The example below (see also \cite{R1}) shows that $X$ and $Y$ can be considered in such a way that
 no simple tensor in $S_{{\mathcal L}(X,Y^*)^*}$ is weak$^*$-weak PC of $B_{{\mathcal L}(X,Y^*)^*}$.

\begin{ex}
Consider $L^1([0,1],X)(=L^1[0,1] \hat{\otimes}_{\pi} X)$ with Lebesgue measure on $[0,1]$.
 Suppose $f \otimes x$ is a Namioka point of $L^1([0,1],X)^{**}$, where $f \in S_{L^1([0,1]}$ and $x \in S_X$. Then by Theorem $\ref{stab w*-w pc}$, $f$ is a Namioka point  of $S_{L^1([0,1])^{**}}$. Let $\Omega$ be the Stone space of $L^\infty([0,1])$ and we use the canonical isometry of this space, with the space of continuous functions, $C(\Omega)$. Then $f \in S_{C(\Omega)^*}$ is a Namioka point. Since $B_{C(\Omega)^*}$ is the weak$^*$-closure of finitely supported measures, as $f$ is a Namioka point, we get that $f$ is a norm limit of a sequence of finitely supported measures.
We may assume without loss of generality that $f = \sum\limits_{i=1}^{\infty} \alpha_i t_i \delta_{w_i}$ for a sequence $(w_i)_i \subset \Omega$, scalars
$(\alpha_i)_i \subset [0,1]$ with $\sum\limits_{i=1}^{\infty} \alpha_i=1$ and $(t_i)_i \subset \{\pm\ 1\}$. Now by Lemma 2.1 in \cite{DMR} on the extremal nature of the set of Namioka points, , we get that for some $i$, $w_i$ is a Namioka point and hence $w_i$ is an isolated point of $\Omega$, the Stone space of $L^\infty [0,1]$.
  Since the Lebesgue measure is non-atomic, we get a contradiction.
\end{ex}

However, the example above does not imply that $B_{L^1([0,1],X)^{**}}$ has no Namioka points. Even when $X$ is an infinite dimensional reflexive space, we do not know if $S_{L^1([0,1],X)^{**}}$ has any Namioka points.

\begin{cor}\label{cor1}
Let X and Y be Banach spaces. Let $x_0\in S_{X}$ and $y_0\in S_{Y}$ be such that $x_0\otimes y_0 \in S_{{\mathcal L}(X,Y^*)^{*}}$ is a weak$^*$-strongly extreme point of $B_{{\mathcal L}(X,Y^*)^{*}}$. Then $x_0$ is weak$^*$-strongly extreme point of $B_{X^{**}}$ and $y_0$ is weak$^*$-strongly extreme point of $B_{Y^{**}}$.
\end{cor}

\begin{proof}
Since $x_0\otimes y_0$ is a Namioka point  and extreme point in $B_{{\mathcal L}(X,Y^*)^{*}}$, by \cite[Lemma 5.1]{DMR}, $x_0\otimes y_0$ is extreme point of $B_{X\hat{\otimes}_{\pi} Y}.$ Hence $x_0$ and $y_0$ are extreme points of $B_X$ and $B_Y$ respectively. Rest follows from Theorem $\ref{bb}$ and $\ref{stab w*-w pc}$.

\end{proof}

From \cite[Lemma 5.1]{DMR}, it is known that, for any Banach space $X$,\\ weak$^*$-strext $(B_{X^{**}})\subset $ $X \bigcap$ ext $(B_{X^{**}})$, where weak$^*$-strext $(C)$ denotes the set of weak$^*$-strongly extreme points of $C$. As a consequence, every weak$^*$-strongly extreme point of $B_{X^{**}}$ is preserved extreme point of $B_X$.

\begin{thm}\label{lgthm1}
\cite[Theorem 1.1]{LGCZ}
Let $X$ and $Y$ be Banach spaces with $\mathcal{K}(X,Y^*)$ is separating for $X\hat{\otimes}_{\pi} Y$. If $z$ is a preserved extreme point of $B_{X\hat{\otimes}_{\pi} Y}$, then $z=x\otimes y$ where $x$ and $y$ are preserved extreme points of $B_X$ and $B_Y$ respectively.
\end{thm}

 Taking into account Corollary $\ref{cor1}$ and Theorem $\ref{lgthm1}$, we have following.

\begin{cor}
Let $X$ and $Y$ be Banach spaces with $\mathcal{K}(X,Y^*)$ is separating for $X\hat{\otimes}_{\pi} Y$. If $z\in S_{{\mathcal L}(X,Y^*)^{*}}$ is weak$^*$-strongly extreme point of $B_{{\mathcal L}(X,Y^*)^{*}}$, then $z=x\otimes y$ for some weak$^*$-strongly extreme points $x\in S_X$ and $y\in S_Y$. In other words, \\
 weak$^*$-strext $(B_{{\mathcal L}(X,Y^*)^{*}})$ $\subset$ weak$^*$-strext $(B_{X^{**}})$ $\otimes $ weak$^*$-strext $(B_{Y^{**}})\subset B_X \otimes B_Y$.
\end{cor}

Using technique similar to the ones given  in Proposition $\ref{c1}$, we get the following.

\begin{prop}\label{p1}
Let $Y$ be a finite dimensional Banach space and $y_0\in S_Y$ be unitary element. Then for any Banach space $X$ with $x_0\in S_X$ is a Namioka point, we have $x_0 \otimes y_0$ is Namioka point of $B_{{\mathcal L}(X,Y^*)^{*}}$.
\end{prop}

\section{Stability of points of continuity}

In this section we consider the weaker version of denting points, by considering only points of weak-norm continuity on $S_X$. We first obtain a unit ball analogue of the results from \cite{W}. A connection to the weak$^*$-topology comes from the fact that $x_0\in S_X$ is weak-norm PC in $B_X$ if and only if $x_0$ is weak$^*$-norm PC in $B_{X^{**}}$, see \cite{HL}. In particular we get that $x_0$ is a weak$^\ast$-norm PC in all higher-order even duals of $X$.

\begin{prop}
Let $X$ and $Y$ be Banach spaces. Also let $x_0\in S_X$, $y_0\in S_Y$ and $x_0 \otimes y_0$ be weak-norm PC in $B_{X\hat{\otimes}_{\pi} Y}$. Then $x_0$ is weak-norm PC in $B_X$ and $y_0$ is weak-norm PC in $B_Y$.
\end{prop}

\begin{proof}
Let $\varepsilon >0$. Since $x_0 \otimes y_0$ is weak-norm PC in $B_{X\hat{\otimes}_{\pi} Y}$, there exists a weakly open set
$$W=\{\tau \in B_{(X\hat{\otimes}_{\pi} Y)} : |T_i(\tau) - T_i(x_0 \otimes y_0)| < \gamma, \ i=1,\ldots,p  \}$$
where $T_1,\ldots T_p\in (X\hat{\otimes}_{\pi} Y)^* = \mathcal{L}(X,Y^*)$
and $x_0 \otimes y_0 \in W \subset B(x_0\otimes y_0, \ \varepsilon)$. Consider the weakly open sets
$$W_1=\{x\in B_{X} : |(y_0\circ T_i)(x) - (y_0\circ T_i)(x_0)|<\gamma, \ i=1,\ldots,p \}$$
$$W_2=\{y\in B_{Y} : |(T_i(x_0))(y) - (T_i(x_0))(y_0)|<\gamma, \ i=1,\ldots,p \}$$
in $B_X$ and $B_Y$ respectively. Then $x_0\in W_1$ and $y_0\in W_2$. Moreover, $W_1\otimes \{y_0\}\subset W$ and $\{x_0\}\otimes W_2\subset W$. Therefore, we can conclude $x_0\in W_1\subset B(x_0, \varepsilon)$ and $y_0\in W_2\subset B(y_0, \varepsilon)$. Indeed, let $x\in W_1$. Then $x\otimes y_0 \in W$. Thus $\|x-x_0\|=\|(x\otimes y_0) - (x_0\otimes y_0)\|<\varepsilon$. Hence, $W_1\subset B(x_0, \varepsilon)$ and similarly $W_2\subset B(y_0, \varepsilon)$.
\end{proof}

\begin{cor}
Let $X$ and $Y$ be Banach spaces. Also let $x_0\in S_X$, $y_0\in S_Y$ and $x_0 \otimes y_0$ be weak$^*$-norm PC in $B_{{\mathcal L}(X,Y^*)^{*}}$. Then $x_0$ is weak$^*$-norm PC in $B_{X^{**}}$ and $y_0$ is weak$^*$-norm PC in $B_{Y^{**}}$.
\end{cor}

We recall from \cite{HWW} that a closed subspace $J \subset X$ is said to be an $M$-ideal, if there is a linear projection $P: X^* \rightarrow X^*$ such that $\|x^*\|=\|P(x^*)\|+\|x^*-P(x^*)\|$ for all $x^\ast \in X^*$ and $\mathrm{ker}(P)=J^\bot$. See \cite[Chapter VI]{HWW} for examples of spaces $X,Y$, for which ${\mathcal K}(X,Y)$ is an $M$-ideal in ${\mathcal L}(X,Y)$. It is known that $M$-ideals which are not $M$-summands (called proper $M$-ideals)do not inherit strongly extreme points or denting points, see \cite{HR}.
Our next theorem is a point of continuity version of this nature. Let $\pi: X \rightarrow X/J$ be the quotient map. We recall that under the canonical identifications, $\pi^{\ast\ast}$ is the quotient map on the bidual, $X^{\ast\ast} \rightarrow X^{\ast\ast}/J^{\bot \bot} $
\begin{thm}\label{mres1}
Let $J \subset X$ be a proper  $M$-ideal such that $X/J$ is not reflexive. Let $x_0 \in S_X$ be a Namioka point. Then $\|\pi(x_0)\|=1$ and is a Namioka point of $X/J$. In particular, $J$ does not contain any Namioka points of $X$. Similar conclusion holds for points of weak-norm PC of $S_X$ when $J$ is of infinite codimension.
\end{thm}
\begin{proof} Since $J$ is an $M$-ideal of $X$, we have, $X^{**}= (J^*)^\bot \bigoplus_{\infty} J^{\bot\bot}$ ($ \ell^\infty$-sum) and $(J^*)^\bot$ is infinite dimensional. Suppose $x_0 \in S_X$ is a Namioka point  and $d(x_0,J)=d(x_0, J^{\bot\bot})<1$. Let $x_0 = j_0^{\ast\ast}+ \Lambda_0$ for $j_0^{\ast\ast}  \in J^{\bot\bot}$ and $\Lambda_0 \in (J^\ast)^\bot$ . It is easy to see that $d(x_0,J) = \|\Lambda_0\|<1$. Now since $(J^*)^\bot$ is not reflexive and as Namioka points lie on the unit sphere,  there is a net $\{x^{**}_{\alpha}\}_{\alpha \in \Delta } \subset S_{(J^*)^\bot}$ such that $x^{**}_{\alpha} \rightarrow \Lambda_0$ in the weak*-topology but not weakly. Now $\|j_0^{\ast\ast}+x^{**}_{\alpha}\|= \max \{ \|j_0^{\ast\ast}\|,1\}=1$ and $j_0^{\ast\ast}+x^{**}_{\alpha} \rightarrow x_0$ in the weak$^\ast$-topology but not weakly. This contradiction shows that $d(x_0,J)=1$. That $\pi(x_0)$ is a Namioka point is easy to see.
\vskip .5 em
The statement about weak-norm PC is similarly proved.
\end{proof}

The following corollary is now easy to see using arguments similar to the ones given above. See \cite[Chapter VI]{HWW} for the examples of spaces that satisfy the assumptions of this corollary. In particular, for $1 <p< \infty$, ${\mathcal K}(\ell^p)$ is an $M$-ideal of infinite codimension in ${\mathcal L}(\ell^p)$.
\begin{cor}
Suppose $X$, $Y$ are such that ${\mathcal K}(X,Y^\ast)$ is an $M$-ideal of infinite codimension in ${\mathcal L}(X,Y^\ast)$. Then no compact operator is a weak$^*$-norm PC in $S_{{\mathcal L}(X,Y^\ast)}$.
\end{cor}
We conclude with a version of Theorem $\ref{mres1}$ for weak$^\ast$-strongly extreme points.
\begin{prop}
Let $J \subset X^\ast$ be a proper $M$-ideal. Then $J$ has no weak$^\ast$-strongly extreme points of $B_{X^*}$.
\end{prop}
\begin{proof}
We only indicate the modification in the proof of Theorem $\ref{mres1}$. Suppose $x^*_0 \in J$ is weak$^\ast$-strongly extreme point of $B_{X^*}$. As before, $X^{***} = (J^*)^\bot \bigoplus_{\infty} J^{\bot\bot}$ and the decomposition is non-trivial. Again, we have from Lemma $\ref{l1}$
 that $x^*_0 \in J^{\bot\bot}$ is an extreme point of $B_{X^{***}}$. This contradicts the fact that an extreme point in an $\ell^\infty$-sum must have all the coordinates as unit vectors. Thus $x_0^* \notin J$.
\end{proof}

\section{On density of Namioka points}

Let $X$ be a Banach space. We say that $X$ has Property $(N)$ if
the set of Namioka points  of $B_{X^*}$ is weakly dense in $S_{X^*}$ (see \cite{Ra3}). We note that if a dual space $X^\ast$ has Property $(N)$, then $X$ is reflexive.  For any infinite compact space $\Omega$, if $A \subset C(\Omega)$ is a closed subalgebra, separating points of $\Omega$, it was shown in \cite{Ra3}, if $A$ has Property $(N)$, then $A$ is isometrically and algebraically isomorphic to $c_0(\Gamma)$, the space of functions vanishing at infinity  on a discrete set $\Gamma$.  It is easy to see using Mazur's theorem that the weak and norm closure of a convex set are the same, that if $X$ has property $(N)$, then the dual unit ball is a norm closed convex hull of its extreme points.  In this section we note that Property $(N)$ is not hereditary answering a question raised in \cite{Ra3}. We recall that if weak$^*$ and weak topologies coincide on $S_{X^*}$, by Lemma $\ref{l1}$, we get that, under the canonical embedding of $X \subset X^{**}$, elements of $X^*$ have unique norm-preserving extension. Such spaces are called Hahn-Banach smooth spaces (see \cite{S}) and it was proved in \cite{BR} that being Hahn-Banach smooth is a hereditary property. Similarly if one only assumes that norm attaining functionals of $X^*$ have unique norm-preserving extensions, then by the Bishop-Phepls theorem \cite{H} and Lemma $\ref{l1}$, we see that the set of Namioka points is  norm dense in $S_{X^*}$.

\begin{prop}
Property $(N)$ is not hereditary.
\end{prop}

\begin{proof}
As noted in Section 1,
the set of weak$^*$-denting points of $B_{X^*}$ is contained in  set of Namioka points of $B_{X^*}$. Also every Banach space $X$ can be isometrically embedded into a Banach space $Z$ such that the set of weak$^*$-denting points of $B_{Z^*}$ is norm dense in $S_{Z^*}$ (see \cite[Corollary 2.8]{SM}). Therefore, it follows that every Banach space $X$ can be isometrically embedded into a Banach space $Z$ such that the set of Namioka points of $B_{Z^*}$ is weakly dense in $S_{Z^*}.$ Thus if Property $(N)$ were hereditary, then every Banach space would have Property $N$. It is easy to see that for the space of continuous functions,  $S_{C([0,1])^*}$ has no Namioka points. Hence, Property $(N)$ is not hereditary.
\end{proof}

\begin{rem}\label{rem1}
Unlike some of the investigations here, the main difficulty in dealing with Namioka points in a dual space $X^*$ is that they need not be preserved in the bidual, $X^{***}$ ( under the canonical embedding), see \cite{DMR}. On the other hand if $x^* \in B_{X^*}$ is a weak$^*$-strongly extreme point, by Lemma $\ref{l1}$ we see that $x^*$ is at least an extreme point of $B_{X^{***}}$. Let $X$ be a  Banach space such that under the canonical embedding in $X^{\ast\ast}$ it is an $M$-ideal. These are called $M$-embedded spaces. See \cite[Chapter III, VI]{HWW} 
for several examples of $M$-embedded spaces from function theory and operator theory. It was shown in \cite{R1} that for a non-reflexive $M$-embedded space $X$, any  $x^\ast \in S_{X^\ast}$ continues to be a Namioka point 
 of the unit ball of all higher order odd duals of $X$.
\end{rem}
Thus we have the following variation of Proposition $\ref{c1}$. For a non-reflexive space $X$ and a positive  integer $k>1$, we denote by $X^{(k)}$, the $k$th ordered dual of $X$. We again recall that since $Y$ is finite dimensional, ${\mathcal L}( X^{(k)},Y) = (X^{(k)}\hat{\otimes}_{\pi} Y^*)^* $.  The following theorem is easy to see.
\begin{thm} Let $X$ be a non-reflexive $M$-embedded space and let $Y$ be a finite dimensional space with a unitary $y^\ast_0 \in S_{Y^\ast}$. For any
$x^\ast_0 \in S(X^\ast)$, 
$x^\ast_0 \otimes y^\ast_0$ is a Namioka point  of $S_{{\mathcal L}(X^{(k)},Y)}$ for any even integer $k$.
\end{thm}

\section*{Acknowledgments}
The Part of this work was done when the second author was visiting Shiv Nadar Institute. She would like to express her deep
gratitude to everyone at Shiv Nadar Institute for the warm hospitality provided during her visit. She would also like  to sincerely thank Dr. Priyanka Grover for the opportunity to visit.
The research of T. S. S. R. K. Rao is supported by a 3-year
project ‘Classification of Banach spaces using differentiability’ funded by
the Anusandhan National Research Foundation (ANRF), Core Research Grant, CRG2023-000595. The research of the second author is financially supported by the Institute Postdoctoral Fellowship, NISER Bhubaneswar.
\vskip 1em

{\bf Data Availability} Not applicable.
\vskip 1em

{\bf Declarations}

{\bf Conflict of interest}  The authors have not disclosed any competing interests.

\bibliographystyle{plain, abbrv}

\end{document}